\documentclass[12pt,psamsfonts]{amsart}
\usepackage{amsaddr}
\usepackage{amsmath}
\usepackage{amsthm}
\usepackage{amssymb}
\usepackage{amscd}
\usepackage{amsfonts}
\usepackage{amsbsy}
\usepackage{graphicx}
\usepackage[dvips]{psfrag}
\usepackage{array}
\usepackage{color}
\usepackage{epsfig}
\usepackage{url}
\usepackage{overpic}
\usepackage{epstopdf}
\usepackage{float}

\newcommand{\R}{\ensuremath{\mathbb{R}}}

\newcommand{\C}{\ensuremath{\mathbb{C}}}

\newcommand{\CC}{\mathcal{C}}

\newcommand{\CR}{\ensuremath{\mathcal{R}}}

\newcommand{\CO}{\ensuremath{\mathcal{O}}}

\newcommand{\ov}{\overline}
\newcommand{\la}{\lambda}
\newcommand{\g}{\gamma}
\newcommand{\G}{\Gamma}

\newcommand{\f}{\varphi}
\newcommand{\al}{\alpha}
\newcommand{\be}{\beta}

\newcommand{\sgn}{\mathrm{sign}}
\newcommand{\de}{\delta}

\newcommand{\re}{\mathrm{Re}}

\def\p{\partial}
\def\e{\varepsilon}

\newtheorem {theorem} {Theorem}
\newtheorem {definition} {Definition}

\newtheorem {remark} [theorem]{Remark}

\newtheorem {mtheorem} {Theorem}

\begin{document}
\renewcommand{\arraystretch}{1.5}

\title[Shilnikov problem in Filippov dynamical systems]
{Shilnikov problem in Filippov\\ dynamical systems}

\author[D.D. Novaes and M.A. Teixeira]
{Douglas D. Novaes and Marco A. Teixeira}

\address{Departamento de Matem\'{a}tica, Universidade
Estadual de Campinas,\\ Rua S\'{e}rgio Buarque de Holanda, 651, Cidade Universit\'{a}ria Zeferino Vaz, 13083-859, Campinas, SP,
Brazil} \email{ddnovaes@ime.unicamp.br} \email{teixeira@ime.unicamp.br}

\subjclass[2010]{34A36,37C29,37H20}

\keywords{Filippov systems, sliding dynamics, sliding homoclinic orbits, Shilnikov homoclinic orbits, sliding Shilnikov orbits, piecewise linear differential systems}

\maketitle

\begin{abstract}
In this paper we introduce the concept of sliding Shilnikov orbits for $3$D Filippov systems. In short, such an orbit is a piecewise smooth closed curve, composed by Filippov trajectories, which slides on the switching surface and connects a Filippov equilibrium to itself, namely a pseudo saddle-focus. A version of the Shilnikov's Theorem is provided for such systems. Particularly, we show that sliding Shilnikov orbits occur in generic one-parameter families of Filippov systems, and that arbitrarily close to a sliding Shilnikov orbit there exist countably infinitely many sliding periodic orbits. Here, no additional Shilnikov-like assumption is needed in order to get this last result. In addition, we show the existence of sliding Shilnikov orbits in discontinuous piecewise linear differential systems.  As far as we know, the examples of Fillippov systems provided in this paper are the first exhibiting such a sliding phenomenon.
\end{abstract}

\bigskip

\noindent{\bf
Piecewise smooth differential systems has become a frontier between mathematics and sciences in general. The study of such systems contributes to the understanding of a wide range of phenomena in several areas of physics, engineering, biology, economy, etc  \cite{BBCK,physDspecial}. The dynamics of piecewise smooth differential systems is ruled by the Filippov convention \cite{F}. In this case, they are called Filippov systems. In such a context, the trajectories are allowed to slide on the switching surface giving rise to ``sliding phenomena'', which are inherent of Filippov systems. In this paper, a study is carried out on a nonlinear sliding phenomenon that we call a ``Shilnikov sliding orbit''. This phenomenon bears a striking resemblance to ``Shilnikov homoclinic orbits'' for smooth differential systems \cite{S1,S2,S3}.  Our main result states a version of Shilnikov's Theorem  for such orbits. More specifically, we show that arbitrarily close to a sliding Shilnikov orbit there exist countably infinitely many sliding periodic orbits. In the smooth context, this result is true under a certain assumption (Shilnikov condition). Here, no additional Shilnikov-like assumption is needed. Finally, we analyze a family of piecewise linear differential systems and we analytically show that a Shilnikov sliding orbit exists for such a family.}

\section{Introduction}

The study of piecewise smooth dynamical systems produces interesting mathematical challenges (see, for instance, \cite{T,NTZ}). These systems are widely used in several branches of science (see, for instance, \cite{BBCK,physDspecial} and the references therein). The present work focuses on the analysis of a nonlinear phenomenon that occurs in such systems, which bears a striking resemblance to {\it Shilnikov homoclinic orbit} in smooth differential systems.

We start by defining the concept of {\it Shilnikov homoclinic orbit} for smooth vector fields. Consider a smooth three dimensional vector field $X$ for which $p\in\R^3$ is a hyperbolic saddle-focus equilibrium. The hyperbolic saddle-focus has a two-dimensional invariant manifold $W^2,$ associated with the complex conjugate eigenvalues, $\la_{1,2}\in\C,$ and a one-dimensional invariant manifold $W^1,$ associated with the  real eigenvalue $\la_0\in\R.$ These two invariant manifolds have opposite stability.  
A {\it Shilnikov homoclinic orbit} $\Gamma$ is a trajectory of $X$ connecting $p$ to itself, bi-asymptotically. Thus, $\Gamma\subset W^1\cap W^2.$ Under suitable genericity conditions, this orbit is a codimension one scenario and its unfolding depends on the {\it saddle quantity} $\sigma=\la_0+\re(\la_{1,2}).$ We say that the {\it Shilnikov condition} {\bf (S)} holds when, among the eigenvalues, the pair of complex conjugate $\la_{1,2},$ is the nearest to the imaginary axis in the complex plane. Notice that, assuming $\la_0>0$ (resp. $\la_0<0$), the Shilnikov condition holds if, and only if, $\sigma>0$ (resp. $\sigma<0$).
In \cite{S2}, Shilnikov showed that  $\Gamma$ is isolated from periodic orbits provided that {\bf (S)} does not hold. In this case, at most one limit cycle bifurcates from $\Gamma$ when it is unfolded. Conversely in \cite{S1}, assuming condition {\bf (S)},  Shilnikov proved the existence of countably infinitely many saddle periodic orbits in a neighborhood of $\Gamma$ and, in \cite{S3}, Shilnikov found a chaotic dynamics near $\Gamma$ (see also \cite{T1,T2}). The interested reader is referred to \cite{survey}, where it can be found a very nice review of Shilnikov's contributions.
  
In the theory of nonsmooth dynamical systems, the notion of solutions of a piecewise smooth differential equation expressed as
 \begin{equation}\label{eq1intro}
x'=Z(x)=F_0(x)+\sgn(h_1(x))F_1(x)+\cdots+\sgn(h_k(x))F_k(x)
\end{equation}
 is stated by the Filippov's convention (see \cite{F}). In the above differential equation, $F_i,$ $i=0,1,\ldots,k,$ are smooth vector fields defined on an open subset $D\subset \R^3,$ and $h_i:D\rightarrow\R,$ $i=1,2,\ldots,k,$ are smooth real functions having $0$ as a regular value and satisfying $h_i^{-1}(0)\cap h_j^{-1}(0)=\emptyset,$ for $i\neq j.$  As usual, $\Sigma=\cup_{i=1}^k h_i^{-1}(0)$ denotes the {\it switching surface}.
 
It is worthwhile to mention that  Shilnikov homoclinic orbits have already been considered in the nonsmooth context. Indeed, in the earlier work of Tresser \cite{T2}, it is mentioned how to extend the Shilnikov's Theorems for Lipschitz continuous piecewise smooth differential systems. In \cite{Med}, the Shilnikov homoclinic  bifurcation was analytically studied in Chua's circuit model, which is a continuous piecewise linear differential system with three pieces.  For this last system, the Shilnikov homoclinic connection and the associated strange attractors had already been numerically detected in \cite{gribov} and \cite{chua1,chua3,chua2}, respectively.   Numerical arguments were also used in \cite{arneodo} to show the existence of Shilnikov homoclinic orbits in continuous piecewise linear differential systems with two pieces. Finally, in \cite{Carmona,LPT}, it was analytically shown the existence of Shilnikov homoclinic orbits for continuous piecewise linear differential systems with two pieces.  We emphasize that all above cited works deal with continuous piecewise smooth differential systems that admit a Shilnikov homoclinic orbit $\Gamma$ satisfying two main properties: {\bf (a)} $\Gamma$ is transversal to $\Sigma$ and {\bf (b)} the hyperbolic saddle-focus equilibrium is not contained in $\Sigma,$ that is, the vector field is smooth in a neighborhood of the equilibrium. 

Under assumptions {\bf (a)} and {\bf (b)}, Shilnikov homoclinic orbits can also be considered for discontinuous piecewise smooth differential systems, which will be referred as  {\it crossing Shilnikov orbit}. In this case, the transversality between $\Gamma$ and $\Sigma$ implies that the dynamics in a neighborhood of $\Gamma,$ concerning the transition of the trajectories of \eqref{eq1intro} through the switching surface $\Sigma,$ is of crossing type. This means that the local trajectories of \eqref{eq1intro}, for points in $\Sigma,$ are trivially given by the concatenation of the trajectories defined in both sides of $\Sigma.$ So, in this neighborhood, the trajectories of \eqref{eq1intro} define a continuous flow. Moreover, since the vector field is smooth in a neighborhood of the equilibrium, it is expected to get similar results to those for continuous vector fields, where the unfolding of $\Gamma$ depends on the Shilnikov condition {\bf (S)}.

In the Filippov context, special attention must be paid to some minimal sets contained in the switching surface $\Sigma,$ for which one cannot find their analogous in the smooth theory, the so called {\it pseudo equilibrium}. A pseudo equilibrium is a proper equilibrium of the well known {\it Filippov sliding dynamics} defined on the switching surface (see Section \ref{sp} for a formal definition of the {\it sliding vector field} and {\it pseudo equilibrium}). The sliding dynamics gives rise to the definition of  {\it sliding homoclinic orbit}, which is a trajectory, in the Filippov sense, sliding through the switching surface and connecting an equilibrium or a pseudo-equilibrium to itself.  In \cite{Glendinning18}, Glendinning studied Shilnikov chaos emerging from sliding homoclinc orbits connecting a hyperbolic saddle-focus to itself. It was also shown that this kind of orbit bifurcates from some boundary equilibrium (see also \cite{Simpson18}). Here, we focus on the study of sliding homoclinic orbits connecting a hyperbolic {\it pseudo saddle-focus} to itself, which we call a {\it sliding Shilnikov orbit} (see Definition \ref{defshil}). The hyperbolic pseudo saddle-focus also has a two-dimensional invariant manifold $W^2 \subset\Sigma^s,$ associated with complex conjugate eigenvalues, and a one-dimensional stable invariant manifold $W^1.$ However, we shall see that the trajectories on $W^1$ reach the pseudo equilibrium in finite time (see Figure \ref{slidingshil}).

\begin{figure}[h]
\begin{center}
\begin{overpic}[width=10cm]{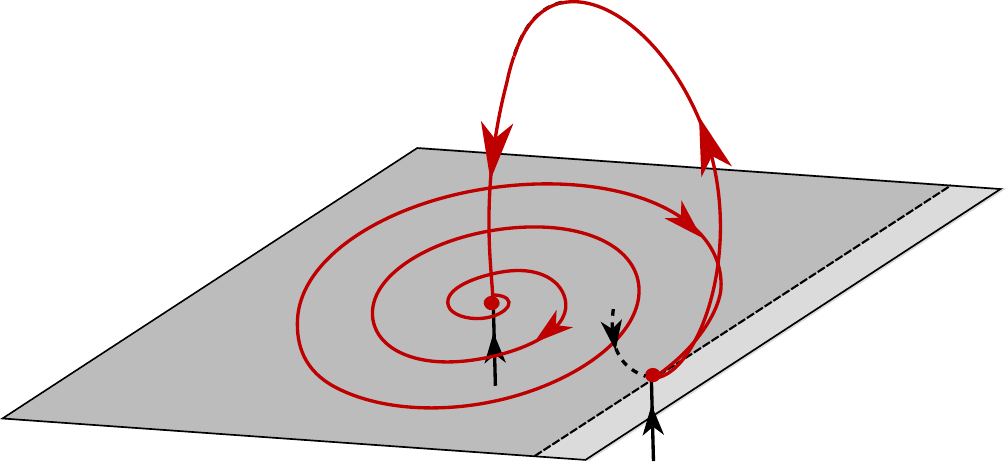}
\put(66,42){$\G$}
\put(51,-3){$\ell$}
\put(50.5,15){$p$}
\put(62.5,5){$q$}
\put(10,6){$W^2$}
\put(51,35){$W^1$}
\put(35,30){$\Sigma$}
\end{overpic}
\end{center}

\caption{Sliding Shilnikov homoclinic orbit $\Gamma$  connecting a hyperbolic pseudo saddle-focus $p$ to itself. Notice that $\Gamma$ has an entire segment of orbit contained in the switching surface $\Sigma$ and leaves it  through the fold point $q.$ Generically, the existence of a fold point $q$ implies the existence of a whole a curve $\ell\subset\Sigma$ of fold points.}\label{slidingshil}
\end{figure}

We emphasize that sliding Shilnikov orbits differ from Shilnikov homoclinic orbits, previously studied for nonsmooth differential systems, mainly in two aspects: first, for sliding Shilnikov orbits the pseudo equilibrium is contained in $\Sigma,$ whereas in all the previous cases the equilibrium is not contained in $\Sigma$; and second, analogous to the case addressed in \cite{Glendinning18}, sliding Shilnikov orbits have an entire segment of orbit contained in the switching surface $\Sigma$ and leaves it  through a quadratic contact point (fold point) between the vector field and $\Sigma,$ whereas the crossing Shilnikov orbits are transversal to $\Sigma.$ 

This paper has two main goals. The first consists in providing a version of the Shilnikov's Theorem, regarding the existence of countably infinitely many periodic solutions, for Filippov systems admitting a sliding Shilnikov orbit $\Gamma.$  Second, we prove that sliding Shilnikov orbits exist in discontinuous piecewise linear differential systems by providing explicit examples. 

The existence of countably infinitely many periodic solutions for Filippov systems admitting a sliding Shilnikov orbit will be obtained via topological mechanisms. More specifically, we shall apply the {\it Brouwer Fixed Point Theorem} for the first return map of $Z$ associated with $\G,$ which is a one-dimensional  map defined in a curve of fold points on the switching surface $\Sigma.$  We shall see that, in this case, no Shilnikov-like assumption is needed and that the hyperbolicity assumption on the pseudo-equilibrium can also be avoided (see Remark \ref{shilcond}). In addition, as far as we know, the examples of discontinuous piecewise linear differential systems provided in this paper are the first exhibiting such a sliding phenomenon.

This paper is organized as follows. Section \ref{sp} contains some basic notions and definitions on Filippov systems as well as the formal definition of a sliding Shilnikov orbit (see Definition \ref{defshil}). In Section \ref{mr}, we state our main results, Theorems \ref{t1} and \ref{t2}.  In short,   Theorem \ref{t1} claims that sliding Shilnikov orbits occur in generic one-parameter families of vector fields in  $\Omega^r.$ In addition, any neighborhood of a Filippov system admitting a sliding Shilnikov orbit contains infinitely many topological equivalence class of vector fields.  Theorem \ref{t2} provides the existence of countably infinitely many sliding periodic orbits  near a sliding Shilnikov orbit. Section \ref{proof} is devoted to the proofs of Theorems \ref{t1} and \ref{t2}. In Section \ref{lm}, we analyze some explicit examples of one-parameter families of piecewise linear differential systems, $Z_{\mu}\in\Omega^r.$ Theorem \ref{t3} shows that, for the critical value of the parameter $\mu=0,$ $Z_0$ exhibits a sliding Shilnikov orbit. Finally, Section \ref{cfd} contains some closing remarks.

\section{Filippov systems and sliding Shilnikov orbit}\label{sp}
 
In this section the basic notions of Filippov systems and the definition of sliding Shilnikov orbit are given. Let $U$ be an open bounded subset of $\R^3.$ We denote by $\CC^r(K,\R^3),$ $K=\ov U,$ the set of all $\CC^r$ vector fields $X:K\rightarrow \R^3$ endowed with the topology induced by the norm $||X||_r=\sup\{||D^i X(\xi)||:\,\xi\in K,\,i\in\{0,1,\ldots,r\}\}.$   Given $h:K\rightarrow\R$ a differentiable function having $0$ as a regular value we denote by $\Omega_h^r(K,\R^3)$ the space of piecewise vector fields
\begin{equation}\label{omega}
Z(\xi)=\left\{\begin{array}{l}
X(\xi),\quad\textrm{if}\quad h(\xi)>0,\\
Y(\xi),\quad\textrm{if}\quad h(\xi)<0,
\end{array}\right.
\end{equation}
with $X,Y\in \CC^r(K,\R^3).$ As usual, system \eqref{omega} is denoted by $Z=(X,Y)$ and the switching surface $h^{-1}(0)$ by $\Sigma.$ So, we are taking $\Omega_h^r(K,\R^3)=\CC^r(K,\R^3)\times \CC^r(K,\R^3)$ endowed with the product topology. When the context is clear we shall refer the sets $\Omega_h^r(K,\R^3)$ and $\CC^r(K,\R^3)$ only by $\Omega^r$ and  $\CC^r,$ respectively. We emphasize that \eqref{omega} is a local description of \eqref{eq1intro}.

Some regions on $\Sigma$ must be distinguished. The points on $\Sigma$ where both vectors fields $X$ and $Y$ simultaneously point outward or inward from $\Sigma$ define, respectively, the {\it escaping} $\Sigma^e$ or {\it sliding} $\Sigma^s$ regions,  and the interior of its complement in $\Sigma$ defines the {\it crossing region} $\Sigma^c.$ The complementary of the union of those regions  constitutes the {\it tangency} points between $X$ or $Y$ with $\Sigma.$ The points in $\Sigma^c$ satisfy $Xh(\xi)\cdot Yh(\xi) > 0,$ where $Xh$ denotes the derivative of the function  $h$ in the direction of the vector $X,$ that is, $Xh(\xi)=\langle \nabla h(\xi), X(\xi)\rangle.$ The points in $\Sigma^s$ (resp. $\Sigma^e$) satisfy $Xh(\xi)<0$ and $Yh(\xi) > 0$ (resp. $Xh(\xi)>0$ and $Yh(\xi) < 0$). Finally, the tangency points of $X$ (resp. $Y$) satisfy $Xh(\xi)=0$ (resp. $Yh(\xi)=0$). 

\begin{definition}\label{fold}
A tangency point $\xi\in\Sigma$ is called {\bf visible fold} of 
$X$ $($resp. $Y)$ if $(X)^2h(\xi)>0$ $($resp. $(Y)^2h(\xi)<0).$ Reversing the inequalities, the tangency point is called {\bf invisible fold}. A visible/invisible fold point $\xi\in\Sigma$ of $X$ (resp. $Y$) is called {\bf visible/invisible fold-regular point} if $Yh(\xi)>0$ (resp. $Xh(\xi)<0$). 
\end{definition}

On the region $\Sigma^s\cup\Sigma^e,$ we define the {\it sliding vector field}
\begin{equation}\label{slisys}
\widetilde Z(\xi)=\dfrac{Y h(\xi) X(\xi)-X h(\xi) Y(\xi)}{Y h(\xi)- Xh(\xi)}.
\end{equation}

 The local trajectory $\f_Z(t,\xi)$ of the discontinuous piecewise differential system $\dot \xi= Z(\xi)$ passing through a point $\xi\in\R^3$ is given by Filippov's convention (see \cite{F,GST}). Let $I\subset\R$ be a sufficiently small neighborhood of the origin. Denote by  $\f_W(t,\xi)$ the trajectory of a vector field $W$ satisfying  $\f_W(0,\xi)=\xi.$ Then, the Filippov convention is summarized as following:

\begin{itemize}
\item[$(i)$] for $\xi\in U$ such that $h(\xi)>0$ (resp. $h(\xi)<0$), the local trajectory of $Z$ is defined as $\f_Z(t,\xi)=\f_X(t,\xi)$ (resp. $\f_Z(t,\xi)=\f_Y(t,\xi)$) for $t\in I.$

\item[$(ii)$] for $\xi\in\Sigma^c$ such that $(Xh)(\xi),(Yh)(\xi)>0$ the local trajectory of $Z$ is defined as $\f_Z(t,\xi)=\f_Y(t,\xi)$ for $t\in I\cap\{t<0\}$ and $\f_Z(t,\xi)=\f_X(t,\xi)$ for $t\in I\cap\{t>0\}.$ For the case $(Xh)(\xi),(Yh)(\xi)<0,$ the definition is the same reversing time;
 
\item[$(iii)$] for $\xi\in\Sigma^s$ the local trajectory is defined as $\f_Z(t,\xi)=\f_{\widetilde{Z}}(t,\xi)$ for $t\in I\cap\{t\geq 0\}$  and $\f_Z(t,\xi)$ is either $\f_X(t,\xi)$ or $\f_Y(t,\xi)$ or $\f_{\widetilde{Z}}(t,\xi)$ for $t\in I\cap\{t\leq 0\}.$ For the case $\xi\in\Sigma^e$ the definition is the same reversing time.
\end{itemize}

For tangency points  $\p\Sigma^c\cup\p\Sigma^s\cup\p\Sigma^e$ the definition of local trajectory is more delicate. Here, we provide the definition for visible fold-regular points (see Definition \ref{fold}). Let $\xi\in\Sigma$ be a visible fold-regular point of $Z.$ Without loss of generality, assume that $\xi$ is a visible fold point of $X.$  Following item $(iv)$ above, the local trajectory of $Z$ passing through $\xi$ is defined as $\f_Z(t,\xi)=\f_1(t,\xi)$ for $t\in I\cap\{t\leq 0\}$ and $\f_Z(t,\xi)=\f_2(t,\xi)$ for $t\in I\cap\{t\geq 0\},$ where $\f_1$ is either $\f_X$ or $\f_Y$ or $\f_{\widetilde{Z}}$ and $\f_1$ is $\f_X.$ 

An equilibrium $\xi^*\in\Sigma^{s,e}$ of the sliding vector field (that is, $\widetilde Z(\xi^*)=0$) is called a {\it pseudo equilibrium} of $Z.$ We say that $\xi^*$ is {\it hyperbolic pseudo equilibrium} of $Z$ when it is a hyperbolic equilibrium of $\widetilde Z.$ Particularly, if $\xi^*\in\Sigma^s$ (resp. $\xi^*\in\Sigma^e$) is an unstable (resp. stable) hyperbolic focus of $\widetilde Z,$
then we call $\xi^*$ a {\it hyperbolic saddle-focus pseudo equilibrium} or just {\it hyperbolic pseudo saddle-focus}.

In order to study the orbits of the sliding vector field it is convenient to define the normalized sliding vector field
\begin{equation}\label{norslisys}
\widehat Z(\xi)=(Y h(\xi)- Xh(\xi))\widetilde Z(\xi)=Y h(\xi) X(\xi)-X h(\xi) Y(\xi),
\end{equation}
which has the same phase portrait of $\widetilde Z$ reversing the direction of the flow in the escaping region. Indeed, system \eqref{norslisys} is obtained by multiplying the sliding vector field \eqref{slisys} (time rescaling) by the function $Y h(\xi)- Xh(\xi),$ which is positive (resp. negative) for $\xi\in\Sigma^s$ (resp. $\xi\in\Sigma^e$). 

The next definition introduces the concept of sliding Shilnikov orbit (see \cite{NovaesThesis15}).

\begin{definition}\label{defshil}
Let $Z=(X,Y)$ be a piecewise continuous vector field having a hyperbolic pseudo saddle-focus $p\in \Sigma^{s}$ $($resp. $p\in \Sigma^{e}),$ and let  $q\in\p\Sigma^s$ $($resp. $q\in\p\Sigma^e)$ be a visible fold-regular point of $Z$ such that:
\begin{itemize}
\item[$(j)$] the backward (resp. forward) trajectory of $Z$ starting at $q$ follows the sliding vector field $\widetilde Z$ and converges to $p$ backward in time $($resp. forward in time$)$;

\item[$(jj)$] the forward (resp. backward) trajectory of $Z$ starting at $q$ intersects the switching surface only at crossing points and reaches $p$ in finite time $t_0>0$ $($resp. $t_0<0).$
\end{itemize}
So, through $p$ and $q,$ a sliding loop $\G$ is easily characterized. We call $\G$ a {\it sliding Shilnikov orbit} $($see Figures \ref{slidingshil}$).$
\end{definition}

The next definition introduces the concept of $\Sigma$-equivalence of Filippov vector fields (see, for instance, \cite{T}).  Of course, the notion of  $\Sigma$-structural stability in $\Omega^r$ is naturally obtained.

\begin{definition}
Let $Z_1,Z_2\in\Omega^r.$ We say that $Z_1$ and $Z_2$ are $\Sigma$-equivalent if there exists a homeomorphism $h:U\rightarrow U$ satisfying $h(\Sigma)=\Sigma$ and sending orbits of $Z_1$ to orbits of $Z_2.$  
\end{definition}

\begin{remark}\label{transv} Assume that $q\in\p\Sigma^{e,s}$ is a visible fold-regular point of $Z.$ Then, the following properties hold (for more details, see \cite{T}): 
\begin{itemize}
\item[$(i)$] there exists a neighborhood $U$ of $q$ such that $\ell=U\cap\p\Sigma^{e,s}$ is constituted by visible fold-regular points; 
\item[$(ii)$] the sliding vector field $\widetilde Z$ \eqref{slisys} is transverse to $\ell;$ 
\item[$(iii)$] and there exists a neighborhood $V$ of $\ell$ such that $Z\big|_V$ is structurally stable.  
\end{itemize}
\end{remark}

\section{Main results}\label{mr}

Our first main result shows that sliding Shilnikov orbits occur in generic one-parameter families of vector fields in  $\Omega^r$ (see \cite{soto}). Furthermore, if $Z_0$ admits a sliding Shilnikov orbit, then any neighborhood $W\subset\Omega^r$ of $Z_0$ contains infinitely many topological equivalence classes of vector fields.

\begin{mtheorem}\label{t1}
Assume that $Z_0=(X_0,Y_0)\in\Omega^r$ $($with $r\geq 1)$ has a sliding Shilnikov orbit $\G_0$ and let $W\subset \Omega^r$ be a small neighborhood of $Z_0.$ Then, there exists a $\CC^1$ function $g:W\rightarrow\R,$ having $0$ as a regular value, such that $Z\in W$ has a sliding Shilnikov orbit $\G$ if, and only if, $g(Z)=0.$ Furthermore, any neighborhood $W\subset\Omega^r$ of $Z_0$ contains infinitely many $\Sigma$-equivalence classes of Filippov vector fields.
\end{mtheorem}

Theorem \ref{t1} is proved in Section 4.1.

\begin{remark}\label{rem}
As a consequence of Theorem \ref{t1}, $g^{-1}(0)$ is a codimension-1 submanifold of $W.$ Hence, for each $Z^*\in g^{-1}(0),$ there exists a curve $Z^*_{\mu}\subset W,$ with $\mu\in\R$ taken in a small neighborhood of $0,$ which intersects $g^{-1}(0)$ transversally at $Z^*_{0}=Z^*.$ Particularly, for $Z_0,$ we say that $Z_{\mu}$ is a splitting of the sliding Shilnikov orbit $\G_0$ (see Figure \ref{unfold}). 
\end{remark}

\begin{figure}[H]
\begin{center}
\begin{overpic}[width=14cm]{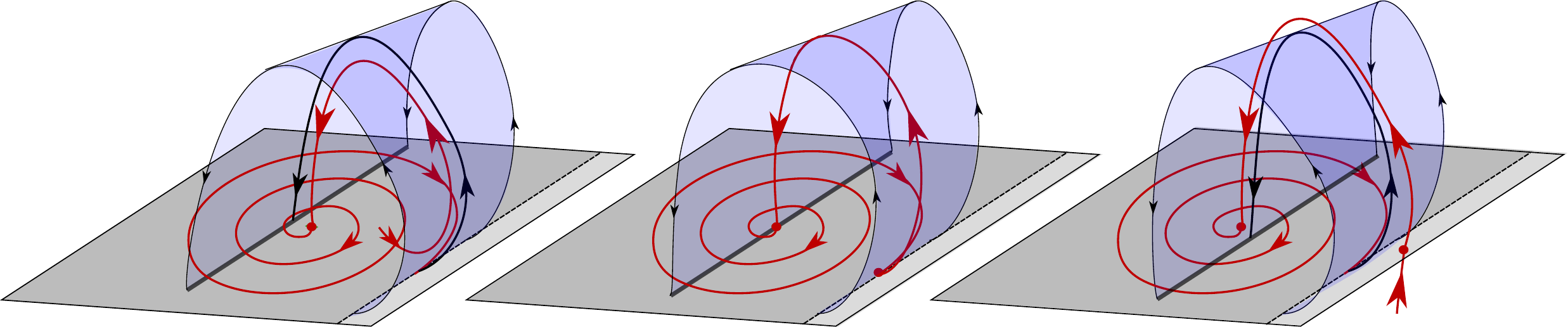}
\put(58,15){$\G$}
\put(10,-3.5){$\mu<0$}
\put(40,-3.5){$\mu=0$}
\put(70,-3.5){$\mu>0$}
\end{overpic}
\end{center}

\vspace{0.8cm}

\caption{Generic one-parameter family $Z_{\mu}=(X_{\mu},Y_{\mu})\in\Omega^r$ for which $Z_0$ has a sliding Shilnikov orbit $\G.$}\label{unfold}
\end{figure}

Our second main result is a version of Shilnikov's theorem regarding the existence of sliding periodic orbits of $Z_{\mu},$ which intersect the fold-regular curve just at one point ($1$-periodic) in a neighborhood of a sliding Shilnikov orbit.

\begin{mtheorem}\label{t2}
Assume that $Z_0=(X_0,Y_0)\in\Omega^r$ $($with $r\geq0)$ has a sliding Shilnikov orbit $\G_0$ and let $Z_{\mu}=(X_{\mu},Y_{\mu})\in\Omega^r$ be a splitting of $\Gamma_0.$ Then, the following statements hold:
\begin{itemize}
\item[$(a)$] for $\mu=0$ every neighborhood $G\subset\R^3$ of $\G_0$ contains countably infinitely many sliding $1$-periodic orbits of $Z_0$;

\item[$(b)$] Let $G\subset\R^3$ be a sufficiently small neighborhood of $\G_0.$ Then, for each  $|\mu|\neq0$ sufficiently small, $G$ contains at least a finite number $N(\mu)>0$ of sliding $1$-periodic orbits of $Z_{\mu}.$ Moreover, $N(\mu)\to\infty$ when $\mu\to 0.$
\end{itemize}
\end{mtheorem}

Theorem \ref{t2} is proved in Section 4.2.

\begin{remark}\label{shilcond}
In the smooth context the hyperbolic saddle-focus has a two-dimensional invariant manifold $W^2,$ associated with the complex conjugate eigenvalues, and a one-dimensional invariant manifold $W^1,$ associated with the real eigenvalue. These two invariant manifolds have opposite stability. As said before, the Shilnikov condition imposes that, among the eigenvalues, the pair of complex conjugate is the nearest to the imaginary axis in the complex plane. Roughly speaking, this means that the linear effect of $W^1$ is stronger than the spiral effect of $W^2$  on the solutions in a neighborhood of the equilibrium. For instance, if $W^1$ is a stable manifold, then the linear attraction to the equilibrium is stronger than the spiral repulsion from the equilibrium.

In Theorem \ref{t2}, it is worthwhile to notice that no Shilnikov-like condition is needed in order to guarantee the existence of countably infinitely many sliding periodic orbits. In our nonsmooth setting, the pseudo saddle-focus also has a two-dimensional unstable invariant manifold $W^2 \subset\Sigma^s,$ associated with the complex conjugate eigenvalues, and a one-dimensional stable invariant manifold $W^1$ (see Figure \ref{slidingshil}). However, the trajectories on $W^1$ reach the pseudo equilibrium in finite time, which implies that the attraction to the pseudo equilibrium in the $W^1$-direction is, in some sense, infinitely stronger than the spiral repulsion from the pseudo equilibrium. Hence, the balance between the attraction and repulsion effects of the invariant manifolds, required by the Shilnikov condition in the smooth context, is automatically satisfied  for sliding Shilnikov orbits with no further assumptions.

Moreover, it will be clear in the proof of Theorem \ref{t2} that the hyperbolicity assumption on the pseudo saddle-focus is not necessary to get statement $(a)$, which still holds when $p_0$ is an unstable weak focus of the sliding vector field.
\end{remark}

\section{Proofs of the main results}\label{proof}

This section is completely devoted to the proofs of Theorem \ref{t1} (Section 4.1) and Theorem \ref{t2} (Section 4.2).  We start providing some common notions and notations  for both proofs.

First, for the sake of simplicity, we take $h(x,y,z)=z,$ that is, $\Sigma=\{z=0\}.$ Suppose that $Z_0=(X_0,Y_0) \in \Omega^r$ admits a sliding Silnikov orbit $\G_0,$ which connects the hyperbolic pseudo focus-saddle $p_0=(0,0,0)\in\Sigma^s$ to itself and contains the fold-regular point $q_0.$ The case when $p_0\in\Sigma^e$ would follow similarly. Without loss of generality, we assume that $q_0$ is a visible fold point for  $X_0$ and that the arc-orbit of $Z$ connecting $q_0$ to $p_0,$ in this direction, intersects the switching surface $\Sigma$ only at $p_0$ and $q_0.$ 

Now, denote $\gamma_{\e }=\ov{B_{\e }(q_0)\cap\p\Sigma^s},$ where $B_{\e }(q_0)\subset \Sigma$ is the planar ball with center at $q_0$ and radius $\e.$ Notice that, for $\e >0$ sufficiently small, $\gamma_{\e }$ is a curve of fold points and  the sliding vector field $\widetilde Z_0$ (see \eqref{slisys}), defined on $\Sigma^s,$ is transversal to $\gamma_{\e }$ (see Remark \ref{transv}). From Definition \ref{defshil}, $p_0$ is a hyperbolic focus of the sliding vector field $\widetilde Z_0,$ the backward trajectory of $\widetilde Z_0,$ starting at $q_0,$ converges to $p_0,$ and the forward trajectory of $X_0,$ starting at $q_0,$ reaches transversally the switching surface $\Sigma$ at $p_0.$ Hence, the implicit function theorem can be used to show that, for $\e >0$ sufficiently small, the backward trajectories of $\widetilde Z_0,$ starting at points of $\gamma_{\e },$  converge to $p_0,$ and the forward trajectories of $X_0,$ starting at points of $\gamma_{\e },$ reach transversally the switching surface $\Sigma$ in a curve $\nu_{\e }.$ Notice that $p_0\in\nu_{\e }.$

Finally, let $W\subset  \Omega^r$ be a small neighborhood of $Z_0.$  From the structural stability property of fold-regular points (as discussed in Remark \ref{transv}), each $Z=(X,Y) \in W$ admits a fold-regular point $p_Z$ contained in a curve of fold-regular points $\gamma_{\e }^Z$ satisfying $p_Z\to p_0$ and $\gamma_{\e}^Z\to\gamma_{\e}$ as $Z\to Z_0.$ In addition, from differentiable dependence results (see, for instance, \cite[Chapter 6]{LS}) of the trajectories of $X$ and $\widetilde Z$ on the initial conditions and parameters ($Z$ can be seen as a parameter on a Banach space), we conclude the following: the backward trajectories of $\widetilde Z,$ starting at points of $\gamma_{\e }^Z,$ converge to $p_Z;$  the  forward trajectories of $X,$ starting at points of $\gamma_{\e }^Z,$ reach transversally the switching surface $\Sigma$ in a curve $\nu_{\e }^Z;$ and $\nu_{\e}^Z\to \nu_{\e}$ as $Z\to Z_0.$ Notice that, in this case, $Z$ has a sliding Shilnikov orbit if, and only if, $p_Z\in\nu_{\e }^Z.$ In the case that $Z_{\mu}$ is a splitting of the sliding Shilnikov connection $\Gamma_0$ (see Remark \ref{rem}) we shall denote $\g_{\e }^{\mu}=\g_{\e }^{Z_\mu},$ $S_{\e }^{\mu}=S_{\e }^{Z_\mu},$ $\nu_{\e }^{\mu}=\nu_{\e }^{Z_\mu},$ and $p_{\mu}=p_{Z_{\mu}}.$

Now we are ready to prove Theorems \ref{t1} and \ref{t2}.

\subsection{Proof of Theorem \ref{t1}}
We may assume that, in a suitable local coordinate system $(x,y)$ around $p_0 \in \Sigma^s,$ $\nu_{\e }$ is given by $y=0,$ that is, $\nu_{\e }=\{(x,0,0):\,-\e\leq x\leq \e\}.$ So,  for $Z \in W,$ $\nu_{\e }^Z$ is also given as a graph $y=k_{Z}(x)= a_0^Z+a_1^Z x + \CO_2(x),$ with $a_0^Z,\,a_1^Z$ small parameters satisfying $a_1^{Z_0}=a_2^{Z_0}=0.$ 

Denote $p_Z=(x_Z,y_Z,0)$ and define $g: W\rightarrow\R$ by $g(Z)=k_{Z}(x_Z) -y_Z.$ From the initial comments, $g$ is a $\CC^1$ function and $g(Z_0)=0.$ We prove now that $0$ is a regular value of $g,$ that is, the linear map $g'(Z^*):\Omega^r\rightarrow\R$ is surjective for every $Z^*\in g^{-1}(0).$ Let $Z^*\in W$ satisfying $g(Z^*)=0$ and take $V\in\Omega^r.$ The derivative of $g$ at $Z^*$ in the direction $V,$ $g'(Z^*)\cdot V,$ can be computed as

\[
g'(Z^*)\cdot V=\dfrac{d}{d v}g(Z(v))\Big|_{v=0}= \lim_{v\to 0}\dfrac{g(Z(v))-g(Z^*)}{v},
\]
where $Z(v)$ is any smooth curve in $\Omega^r$ such that $Z(0)=Z^*$ and $Z'(0)=V\in\Omega^r.$ So, taking $Z(v)$ in such a way that $p_{Z(v)}=(0,0)$ and $k_{Z(v)}(x)=v,$ we get that $g(Z(v))=v$ and, therefore, $g'(Z^*)\cdot V=1.$ This implies that $g'(Z^*)$ is surjective for every $Z^*\in g^{-1}(0).$

Finally, let $Z_{\mu}$ be a splitting of the sliding Shilnikov connection $\Gamma_0$ (see Remark \ref{rem}). Since the pseudo equilibrium $p_{\mu}$ of $Z_{\mu}$ is not contained in $\nu_{\e }^{\mu},$ for $\mu\neq0,$ the saturation of $\gamma_{\e}^{\mu}$ through the backward flow of $\widetilde Z_0$ intersects $\nu_{\e }^{\mu}$ in a finite number $N(\mu)$ of disjoint sets. Thus, one can find trajectories of $\widetilde Z_0,$ starting at $\gamma_{\e}^{\mu},$ which intersect $\nu_{\e }^{\mu}$ in $N(\mu)$ points. In addition, the intersection between $\nu_{\e }^{\mu}$  and any trajectory of $\widetilde Z_0,$ starting at $\gamma_{\e}^{\mu},$ has no more than $N(\mu)$ points.  Now, if  $Z_1,Z_2\in W$ are topologically equivalent, then $\gamma_{\e}^{Z_1}$ and $\nu_{\e}^{Z_1}$ are sent to $\gamma_{\e}^{Z_2}$  and $\nu_{\e}^{Z_2},$ respectively. Hence, $N(\mu_1)\neq N(\mu_2)$ implies that $Z_{\mu_1}$ and $Z_{\mu_2}$ are not $\Sigma$-equivalent. Since $N(\mu)\to\infty$ as $\mu\to 0,$ we  get the existence of infinitely many $\Sigma$-equivalence classes of Filippov vector fields in any neighborhood $W\subset\Omega^r$ of $Z_0.$ This concludes the proof.

\subsection{Proof of Theorem \ref{t2}}

To prove statement $(a),$  let $S_{\e }$ be the backward saturation of $\g_{\e }$ through the flow of the sliding vector field $\widetilde Z.$  Since $p_0$ is a focus of the sliding vector field, we get
\[
S_{\e }\cap\nu_{\e }=\bigcup_{i=1}^{\infty} I_i,
\]
where the sequence of compact sets $(I_i)_{i=1}^{\infty}$ satisfies: $I_i\cap I_j=\emptyset$ for $i\neq j$ and $I_i\to \{p_0\}$ (see Figure \ref{POShilnikov}).

\begin{figure}[H]
\begin{center}
\begin{overpic}[width=9.5cm]{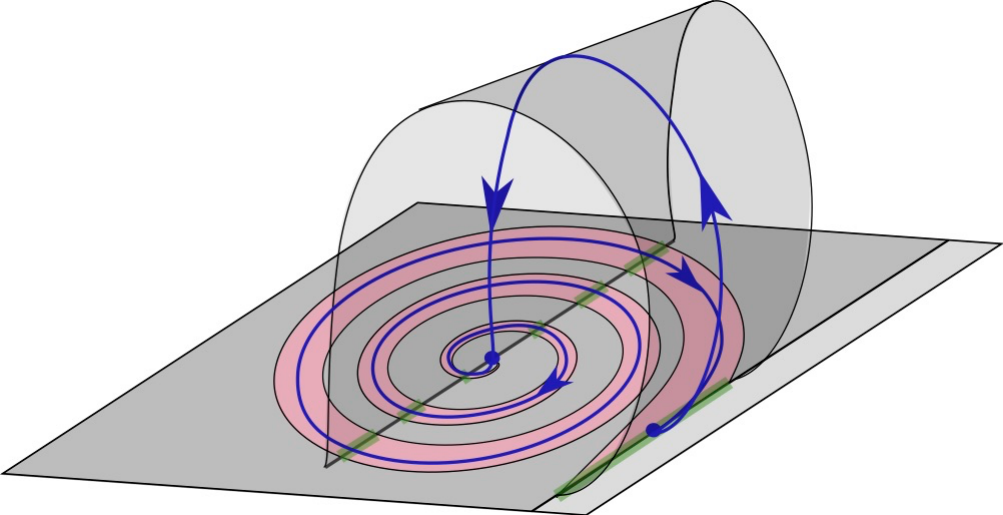}
\put(70,35){$\G$}
\put(83,24){$\Sigma^{s}$}
\put(51,15){$p_0$}
\put(66,6){$q_0$}
\put(70.5,10){$\g_{\e }$}
\put(24,10){$S_{\e }$}
\put(29,4){$\nu_{\e }$}
\put(67,26){$I_1$}
\put(38,6){$I_2$}
\put(57,23.5){$I_3$}
\put(39,12){$I_4 \cdots$}
\end{overpic}
\end{center}

\medskip

\caption{Representation of the sequence of compact sets $(I_i)_{i=1}^{\infty},$ satisfying $I_i\to \{p_0\},$ which are characterized by the intersection between $\nu_{\e }$ and $S_{\e }.$
}\label{POShilnikov}
\end{figure}

Now, for each $i=1,2,\ldots,$ we define $J_i$ as the intersection between the backward saturation of $I_i$ through the flow of $X$ with the curve $\g_{\e }.$ Clearly, $J_i\cap J_j=\emptyset$ for $i\neq j$ and $J_i\to\{q_0\}.$ Moreover, a first return map $\pi$ is well defined on $\cup_{i=1}^{\infty} J_i.$ It is easy to see that $\pi$ is not injective. In what follows, we shall construct a sequence of applications  $(\psi_i)_{i=1}^{\infty},$ $\psi_i:J_i\rightarrow J_i,$ satisfying the following property: 
\begin{itemize}
\item[{\bf (P)}] {\it for each $i\in\mathbb{N},$ if $y \in J_i$ and $x=\psi_i(y),$  then $\pi(x)=y.$}
\end{itemize}
This property implies that a fixed point of $\psi_i$ is also a fixed point of $\pi.$ So, for $\xi\in\Sigma^s$ and $z\in\R^3$ let $\f^s(t,\xi)$ and $\f^X(t,z)$ be the flows of $\widetilde Z$ and $X,$ respectively.
For $\xi\in J_i$ there exists $t^s_i(\xi)<0$ and  $t^X_i(\xi)<0$ such that $\xi_i(\xi)=\f^s(t^s_i(\xi),\xi)\in I_i$ and  $\f^X(t^X_i(\xi),\xi_i(\xi))\in J_i,$ respectively. So, define $\psi_i(\xi)=\f^X(t^X_i(\xi),\xi_i(\xi)).$ Notice that $\psi_i$ is a $\CC^r$ function. From the above construction, the property {\bf (P)} is satisfied for the sequence of functions $(\psi_i)_{i=1}^{\infty}$ and, consequently, fixed points of $\psi_i$ correspond to sliding periodics orbit of $Z$ (see Figure \ref{PPOShilnikov}). 

\begin{figure}[H]
\begin{center}
\begin{overpic}[width=9.5cm]{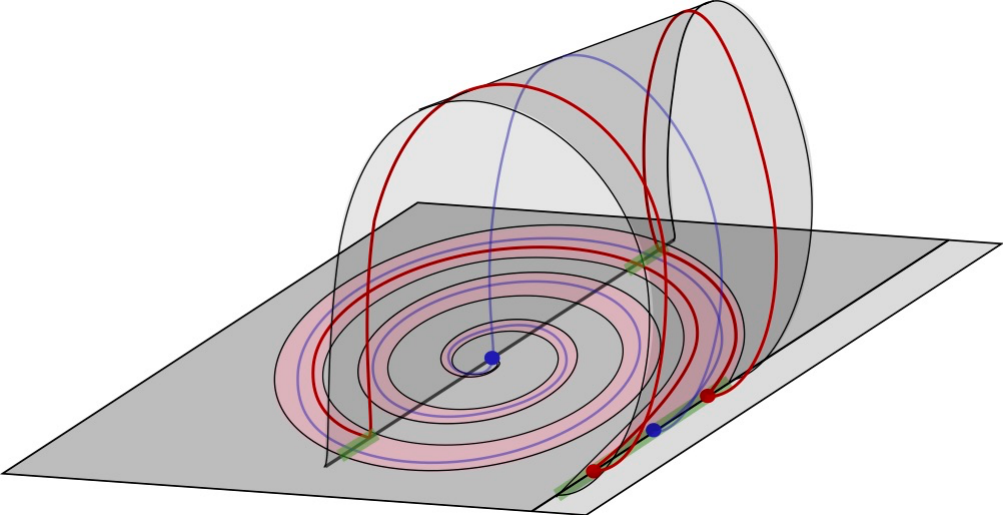}
\put(70,35){$\G$}
\put(83,24){$\Sigma^{s}$}
\put(51,15){$p_0$}
\put(66,6){$q_0$}
\put(70.5,9.5){$q_1$}
\put(59,2){$q_2$}
\put(67,26){$I_1$}
\put(38,6){$I_2$}
\end{overpic}
\end{center}

\medskip

\caption{Representation of sliding periodic orbits of $Z$ corresponding to fixed points of $\psi_i.$}\label{PPOShilnikov}
\end{figure}

Now, for each $i\in\mathbb{N},$ $\psi_i$ is a continuous function from a compact interval $J_i$ to itself. So,  applying the {\it Brouwer fixed-point Theorem} we obtain a sequence $(q_i)_{i=1}^{\infty}$ such that $q_i\in J_i$ and $\psi_i(q_i)=q_i.$ Hence, we conclude that there exists a sequence of sliding periodic orbits of $Z$ passing through $q_i.$ The proof of statement $(a)$ follows just by observing that $q_i\to q_0.$ 

In what follows we prove statement $(b).$ Let $Z_{\mu}$ be a splitting of the sliding Shilnikov connection $\Gamma_0$ (see Remark \ref{rem}). Since the pseudo equilibrium $p_{\mu}$ of $Z_{\mu}$ is not contained in $\nu_{\e }^{\mu},$ the intersection $S_{\e }^{\mu}\cap\nu_{\e }^{\mu}\cap G$ has only a finite number $N(\mu)$ of disjoint sets $I_i.$ Furthermore, the number of disjoint sets $N(\mu)$ in this intersection tends to infinity when $\mu$ goes to $0.$ From here, the proof of statement $(b)$ follows analogously to the proof of statement $(a).$

\section{A piecewise linear example}\label{lm}
In this section we present one-parameter families of discontinuous piecewise linear vector fields $Z_{\mu}\in\Omega^r$ for which $Z_0$ admits a sliding Shilnikov orbit. 

For $\al>0,$ $\be>0,$ and $\mu\in\R,$ consider the following discontinuous piecewise linear vector field.
\begin{equation}\label{s1}
Z_{\mu}(x,y,z)=\left\{\begin{array}{ll}
X(x,y,z)=\left(
\begin{array}{c}
-\al\\
x-\beta\\
y-\dfrac{3\beta^2}{8\al}
\end{array}
\right)&\textrm{if}\quad z>0,\vspace{0.2cm}\\

Y_{\mu}(x,y,z)=\left(
\begin{array}{c}
\al\\
\dfrac{3\al}{\beta}y+\beta\\
\mu+\dfrac{3\beta^2}{8\al}
\end{array}
\right)&\textrm{if}\quad z<0.
\end{array}\right.
\end{equation}
Notice that $\Sigma=\{z=0\}$ is the switching surface of system \eqref{s1}, which can be decomposed as $\Sigma=\ov{\Sigma^c}\cup\ov{\Sigma^s}\cup\ov{\Sigma^e}$ where
\[
\begin{array}{llll}
\Sigma^c=\left\{(x,y,0):\,y>\dfrac{3\be^2}{8\al}\right\}, &\Sigma^s=\left\{(x,y,0):\,y<\dfrac{3\be^2}{8\al}\right\}&\textrm{and}&\Sigma^e=\emptyset.
\end{array}
\]
Moreover, $c=\left(\beta,3\beta^2/(8\al),0\right)\in\p\Sigma^s$ is a cusp-regular point, $\{\left(x,3\beta^2/(8\al),0\right):\, x<\beta\}\subset\p\Sigma^s$ is a curve of invisible fold-regular points, and $\{\left(x,3\beta^2/(8\al),0\right):\, x>\beta\}\subset\p\Sigma^s$ is a curve of visible fold-regular points.

We shall see that, in the above families of Filippov systems, $\mu$ is a bifurcation parameter, for which a sliding Shilnikov orbit exists only for the critical value $\mu=0.$   Parameters $\al$ and $\beta$ do not play any role, and just make the example more general. Finally, we point out that the presence of the fractions $3\beta^2/(8\al)$ and $3\al/\beta$ will greatly simplify the expression of the normalized sliding vector field after a variable rescaling.

 Our main result on the above family is the following.

\begin{mtheorem}\label{t3} 
For each positive real numbers $\al$ and $\beta$ the following statements hold:
\begin{itemize}
\item[$(a)$] For $\mu=0,$  the origin $p_0=(0,0,0)$ is a hyperbolic pseudo saddle-focus of $Z_0,$ which is an unstable hyperbolic focus of the sliding vector field $\widetilde Z_0.$ Moreover, $Z_0$ admits a sliding Shilnikov orbit, connecting $p_0$ to itself, passing through the visible fold-regular point $q_0=\big(3\be/2,3\be^2/(8\al),0\big)$ (see Figure \ref{simula}).
\item[$(b)$] For $\mu\neq0,$ $Z_0$ does not admit a sliding Shilnikov orbit.
\end{itemize}
\end{mtheorem}

\begin{proof}
We compute the sliding vector field \eqref{slisys} as
\[
\widetilde Z_{\mu}(x,y,0)=\left(\dfrac{4\al^2 (y+\mu)}{4\al (y-\mu)-3\beta^2}\,,\,\dfrac{3\be^3x+\al\beta^2y-24\al^2y^2+8\al\be\mu(x-\be)}{6\be^3-8\al\beta (y-\mu)}\,,\,0\right).
\]
Since the sliding vector field $\widetilde Z_{\mu}$ is defined only on the planar region $\Sigma^s\cup\Sigma^e\subset\Sigma,$ it can be identified with the following planar vector field
\begin{equation}\label{slid}
\widetilde Z_{\mu}(x,y)=\left(\dfrac{4\al^2 (y+\mu)}{4\al (y-\mu)-3\beta^2}\,,\,\dfrac{3\be^3x+\al\beta^2y-24\al^2y^2+8\al\be\mu(x-\be)}{6\be^3-8\al\beta (y-\mu)}\right).
\end{equation}

Notice that $p_{\mu}=\left(\dfrac{3\al}{\be}\mu,-\mu \right)$ is a singularity of $\widetilde Z_{\mu}$ for every $\mu\in\R.$ 

For $\mu=0,$ the normalized sliding vector field of \eqref{s1} writes
\begin{equation}\label{nslid}
\widehat Z_{0}(x,y)=\left(-\al y\,,\,\dfrac{3\be^2}{8\al}x+\dfrac{\be}{8}y-\dfrac{3\al}{\be}y^2\right).
\end{equation}
Notice that the origin is a hyperbolic focus for $\widehat Z_0,$ and also for  $\widetilde Z_0.$ Indeed, their eigenvalues are given by
\begin{equation*}\label{eigenvalues}
\la^{\pm}=\dfrac{\al}{12\be}\pm i\dfrac{\sqrt{95}\al}{12\beta}.
\end{equation*}
It implies that $p_0=(0,0,0)$ is a hyperbolic pseudo saddle-focus of $Z_0.$ Moreover, since $\re(\la^{\pm})>0,$ then $(0,0)$ is an unstable hyperbolic focus of the (normalized) sliding vector field \eqref{nslid}.  After a change of variables and a time rescaling, expressed by
\begin{equation}\label{change}
(x,y)=\left(\dfrac{3\be}{2}u,\dfrac{3\be^2}{8\al}v\right)\quad \text{and} \quad t=-\dfrac{4}{\beta}\tau,
\end{equation}
the normalized sliding vector field $\widehat Z_0$ becomes
\begin{equation}\label{trans}
\ov{Z}=\left(v,-6u-\dfrac{1}{2}v+\dfrac{9}{2}v^2\right).
\end{equation}
Notice that the time rescaling \eqref{change} reverses the direction of the flow of $\eqref{slid}.$ The tangency line $\p\Sigma^s$ is given now, in $(u,v)$ coordinates, by $\ell=\{(u,1):u\in\R\}.$

We claim that the orbit of system \eqref{trans} starting at the point $(1,1)\in\ell$ converges to the focus equilibrium $(0,0)$ without touching the line $\ell.$ Clearly, going back through the transformation \eqref{change}, this implies that the orbit of system \eqref{slid} starting at the visible fold-regular point $q_0=\left(3\be/2,3\be^2/(8\al),0\right)\in\p\Sigma^s$ is attracted, now backward in time, to the focus $(0,0,0)$ without touching the tangency line $\p\Sigma^s.$ To prove this claim we shall construct a compact region $\CR$ in the $u,v$-plane, which is positively invariant through the flow of the vector field \eqref{trans}. Accordingly, let $m(v)=-13/108+9v^2/13+54v^3/169,$ and define the curves
\[
\begin{array}{l}
\CC_1=\{(u,1):\,m(1)\leq u\leq 1\},\\
\CC_2=\{(u,-2u+3):\,1\leq u\leq3/2\},\\
\CC_3=\{(3/2,v):\,-91/72<v<0\},\\
\CC_4=\{(u,-91/72):\,m(-91/72)\leq u\leq 3/2\},\\
\CC_5=\{(m(v),v):\,-91/71\leq v\leq 1\}.
\end{array}
\]
We define $\CR$ as being the compact region delimited by the curves $\CC_i$ for $i=1,2,\ldots,5$ (see Figure \ref{sliding}). After some standard computations we conclude that $\CR$ is positively invariant through the flow of \eqref{trans}. Furthermore, the vector field  \eqref{trans} has at most one limit cycle (see Theorem A of \cite{CGL}), which is hyperbolic. So,  from the positive invariance of $\CR,$ from the stability of the equilibrium $(0,0),$ and from the uniqueness of a possible limit cycle we conclude that,  if this limit cycle exists, then it cannot be inside $\CR.$ Applying the Poincar\'{e}-Bendixson theorem we conclude that the stable focus of \eqref{trans} attracts the orbits, forward in time, of all points in $\CR$ without touching the line $\ell.$ The claim follows by noting that $(1,1)\in\CR$ (see Figure \ref{sliding}).
\begin{figure}[H]
\begin{center}
\begin{overpic}[width=7cm]{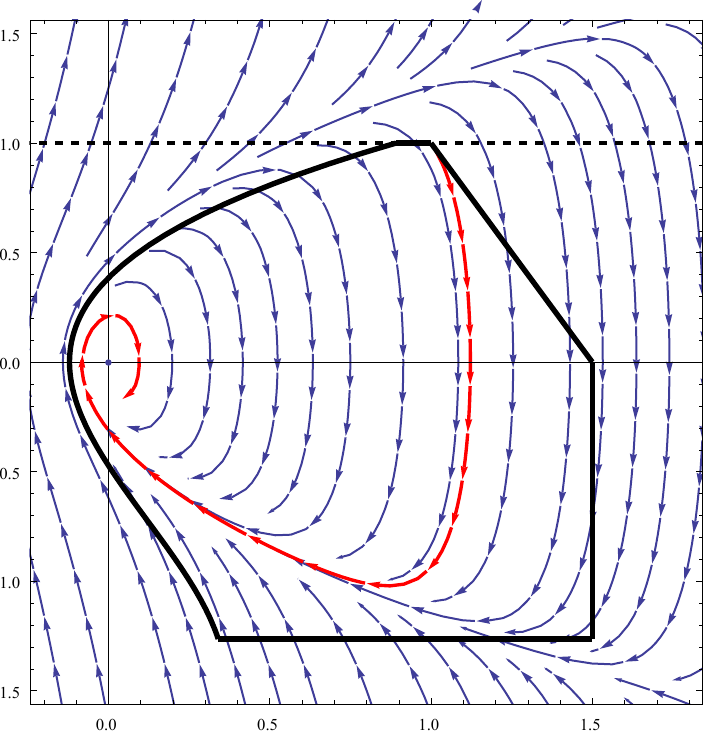}
\put(98,80){$\ell$}
\put(55,82){$\CC_1$}
\put(71,65){$\CC_2$}
\put(82,30){$\CC_3$}
\put(51,7){$\CC_4$}
\put(15,68){$\CC_5$}
\end{overpic}
\end{center}
\caption{The dashed bold line represents the tangency line $\ell.$ The continuous bold lines represents the curves $\CC_i,$ $i=1,2,\ldots,5,$ and delimit the compact region $\CR$ which is positively invariant through the flow of \eqref{trans}. The red curve is the trajectory of \eqref{trans} starting at $(1,1)$ being attracted to the focus $(0,0).$}\label{sliding}
\end{figure}

We remark that the function $m(v),$ defined above, was obtained as an approximation of  an invariant curve of  \eqref{trans} expressed as $u=\ov{m}(v).$ Indeed, taking $\ov {m}(v)=p_3(v)+\CO(v^4),$ with $p_2(v)=k_0+k_1 v+k_2 v^2+k_3 v^3,$ and imposing that 
\[
\langle\nabla(u-\ov{m}(v)), \ov{Z}(u,v)\rangle\big|_{u=\ov{m}(v)}=0,
\]
 we conclude that $p_3(v)=m(v).$ Furthermore, considering the curve $u=m(v)$ we see that
\[
\langle\nabla(u-m(v)), \ov{Z}(u,v)\rangle\big|_{u=m(v)}=\dfrac{729 v^4(91+72 v)}{28561},
\]
which does not change its sign for $-91/71\leq v\leq 1.$  

Finally, the vector field $X$ is linear. Thus, its trajectory starting at the point $q=(x_0,3\be^2/(8\al),0)\in\p\Sigma^s$ is easily computed as
\[
 \f^+(t,q)=\left(x_0-\al\, t\,,\,\dfrac{3\be^2+8\al(x_0-\be)t-4\al^2 t^2}{8\al}\,,\,\dfrac{3(x_0-\be)t^2-\al t^3}{6}\right).
\]
Notice that, for $q\in\p\Sigma^s$ and $t^+(q)=3(x_0-\beta)/\al$>0,  $\f^+(t^+(q),q)\in\Sigma^s.$ Moreover, $\f^+(t^+(q_0),q_0)=p_0.$ It implies that there exists a sliding Shilnikov orbit of $Z_0$ connecting $p_0$ to itself passing through $q_0$ (see Figure \ref{simula}). This concludes the proof of statement $(a).$

To get statement $(b),$ we notice that there is no solution to the equation $\f^+(t^+(q),q)=p_{\la},$ for $\la\neq0.$
\end{proof}

\begin{figure}[H]
\begin{center}
\begin{overpic}[width=13cm]{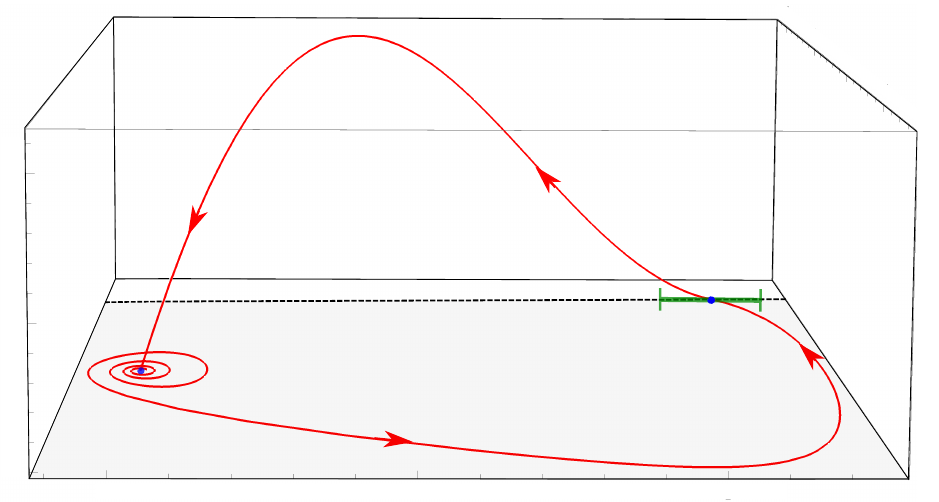}
\put(75,23.5){$q_0$}
\put(12,17){$p_0$}
\put(90,4.5){$\Sigma^s$}
\put(55,40){$\Gamma_0$}
\put(5,21){$\p\Sigma^s$}
\end{overpic}
\end{center}
\caption{Sliding Shilnikov orbit $\Gamma_0$ of the piecewise linear differential system \eqref{s1} (for $\al=1/2,$ $\beta=3/2,$ and $\mu=0$) connecting the hyperbolic pseudo saddle-focus $p_0=(0,0,0)$ to itself, passing through the  visible fold-regular point $q_0=\big(3\be/2,3\be^2/(8\al),0\big).$}\label{simula}
\end{figure}

\section{Conclusions and further directions}\label{cfd}
In this paper we study a sliding homoclinic orbit to a pseudo saddle-focus of Filippov systems. Following the nomenclature of  smooth differential systems, this sliding homoclinic orbit is called sliding Shilnikov orbit. A version of the Shilnikov's Theorem was given in this context. More specifically, Theorem \ref{t1} showed that sliding Shilnikov orbits occur in generic one-parameter families of vector fields in  $\Omega^r$ (see \cite{soto}). In addition, if $Z_0$ admits a sliding Shilnikov orbit, then any neighborhood $W\subset\Omega^r$ of $Z_0$ contains infinitely many topological equivalence classes of vector fields. Theorem \ref{t2} provides the existence of countably infinitely many sliding periodic orbits  near a sliding Shilnikov orbit. In Theorem \ref{t2}, it is worthwhile to mention that no Shilnikov-like condition is needed in order to guarantee the existence of countably infinitely many sliding periodic orbits (see Remark \ref{shilcond}). Finally, Theorem \ref{t3} provided explicit 1-parameter families, $Z_{\mu}\in\Omega^r,$ of piecewise linear vector fields, for which $Z_0$ admits a sliding Shilnikov orbit.

Understanding how a sliding Shilnikov orbit behaves under smoothing process (see \cite{ST}) is a major problem in this context. If $Z\in\Omega^r$ admits a sliding Shilnikov orbit, it seems possible to show the existence of 1-parameter families, $Z^{\de},$ of smooth differential systems approaching continuously to $Z$ such that, for each $\de>0$ small enough, $Z^{\de}$ admits a Shilnikov connection. 

Also, higher dimensional vector fields allow the existence of many other kinds of sliding homoclinic connections. So,  the study of typical sliding homoclinic connection in higher dimensions seems to be a very fertile theme of research.  

Another possible direction for further investigation is to apply the techniques from ergodic theory to provide deeper results for these kind of connections. For instance, results on the existence of symbolic extensions, conjugation with Bernoulli shifts, and existence of Smale horseshoes would be very welcome.

Finally, preliminary studies indicate that a sliding Shilnikov orbit may exist in piecewise smooth biological models, namely prey switching models (see \cite{Piltz}). Since the existence of Shilnikov homoclinic orbits is a usual route to chaos, it seems interesting to investigate the existence of sliding Shilnikov orbits in piecewise smooth models of real phenomena.

\section*{Acknowledgements}

We thank the referees for their comments and suggestions that helped us to greatly improve the presentation of this paper.

DDN is partially supported by FAPESP grant 2018/16430-8 and by CNPq grants 306649/2018-7 and 438975/2018-9. MAT is partially supported by a CNPq grant 301275/2017-3.

\bibliographystyle{abbrv}
\bibliography{references.bib}

\end{document}